\pdfoutput=1 
\documentclass{amsart}

\usepackage{graphicx}

\usepackage{amssymb}
 \usepackage{amsthm,amsmath}

\newtheorem{theorem}{Theorem}[section]
\newtheorem{lemma}[theorem]{Lemma}
\newtheorem{proposition}[theorem]{Proposition}
\newtheorem{corollary}[theorem]{Corollary}
\newtheorem{claim}[theorem]{Claim}

\theoremstyle{definition}
\newtheorem{definition}[theorem]{Definition}
\newtheorem{example}[theorem]{Example}

\theoremstyle{remark}
\newtheorem{remark}[theorem]{Remark}

\newtheorem*{thm_R-group_characterization}{Theorem~\ref{R-group_characterization}}

\begin{document}

\title[Generalized torsion elements of order two in $3$--manifold groups]{Classification of generalized torsion elements of order two in $3$--manifold groups}

\author{Keisuke Himeno}
\address{Graduate School of Advanced Science and Engineering, Hiroshima University,
1-3-1 Kagamiyama, Higashi-hiroshima, 7398526, Japan}
\email{himeno-keisuke@hiroshima-u.ac.jp}

\author{Kimihiko Motegi}
\address{Department of Mathematics, Nihon University,
3-25-40 Sakurajosui, Setagaya-ku, Tokyo, 1568550, Japan}
\email{motegi.kimihiko@nihon-u.ac.jp}

\author{Masakazu Teragaito}
\address{Department of Mathematics Education, Hiroshima University, 
1-1-1 Kagamiyama, Higashi-Hiroshima, 7398524, Japan}
\email{teragai@hiroshima-u.ac.jp}

\begin{abstract}
Let $G$ be a group and $g$ a non-trivial element in $G$. 
If some non-empty finite product of conjugates of $g$ equals to the identity,
then $g$ is called a generalized torsion element. 
The minimum number of conjugates in such a product is called the order of $g$. 
We will classify $3$--manifolds $M$, each of whose fundamental group has a generalized torsion element of order two. 
Furthermore, we will classify such elements in $\pi_1(M)$. 

We also prove that $R$--group and $\overline{R}$--group coincide for $3$--manifold groups, 
and classify $3$--manifold groups which are $R$--groups (and hence $\overline{R}$--groups).  
\end{abstract}

\maketitle








\section{Introduction}\label{sec:intro}

It is well known that the fundamental group of any aspherical $3$--manifold, which is a $K(\pi, 1)$ space, does not admit a torsion element. 
On the other hand, the fundamental group of such a manifold may have a generalized torsion element defined below. 

\begin{definition}
\label{g-torsion}
A non-trivial element $g$ in a group $G$ is a \emph{generalized torsion element} 
if some non-empty finite product of its conjugates is the identity, i.e. 
\[
 (x_1^{-1} g x_1)(x_2^{-1} g x_2) \cdots(x_{k}^{-1}g x_{k})=1\quad \textrm{for some}\  x_{1},\ldots,x_k \in G.
\]
As a natural generalization of the order of torsion element, the \emph{order} of a generalized torsion element $g$ is defined as 
the minimum number of conjugates yielding the identity. 
\end{definition}

If $x_1 = \cdots = x_k = 1$, then $g$ is a usual torsion element. 
Since a generalized torsion element is not the identity, its order is at least two. 
It is known that for a given integer $n > 1$,  
there are infinitely many hyperbolic $3$--manifolds $M_n$, each of whose fundamental group has rank $n$ and contains a generalized torsion element. 
Furthermore, we may choose $M_n$ so that the order of the generalized torsion element is arbitrarily large; see \cite{IMT_BLMS2023}.

In what follows, for notational simplicity, we will write $g^x=x^{-1} g x$ for $g, x \in G$. 
Also, we assume that $3$--manifolds are connected.

A generalized torsion element naturally appears in the fundamental group of the Klein bottle.
It has a presentation $\langle a,b \mid a^{-1}bab = 1 \rangle$.
The relation shows $b^a b=1$, so the generator $b$ is a generalized torsion element of order two.

Recall that every compact orientable $3$--manifold $M$ can be uniquely expressed as a connected sum 
$M = M_1 \# M_2\# \cdots \# M_n$ of $3$--manifolds $M_i$ which are prime in the sense that they can be 
expressed as connected sums only in the trivial way $M_i = M_i\# S^3$ \cite{Kne1929, Milnor1962}. 
Note that an orientable, prime $3$--manifold is irreducible, except for $S^2 \times S^1$. 

Assume that each $\pi_1(M_i)$ is torsion-free. 
Then as shown in \cite[Theorem~1.3 (2)]{IMT_PAMS2019} any generalized torsion element in $\pi_1(M)$ is conjugate into a generalized torsion element of the same order
in $\pi_1(M_i)$ for some $i$. 

As we remarked in \cite{IMT_PAMS2019},
if some $\pi_1(M_i)$ has torsion, then this is not the case. 
See Remark~1.4 in \cite{IMT_PAMS2019}, 
which shows that there can exist a generalized torsion element $g$ of $\pi_1(M)$ not conjugate into any factor group $\pi_1(M_i)$. 
On the other hand, obviously we have a torsion element (by the assumption), but its order is not necessarily the same as the order of $g$.  
However, if we restrict our attention to generalized torsion elements of order two,
then even when a generalized torsion element $g$ of order two is not conjugate into some factor group, 
we may find a torsion element of order two. 

\begin{theorem}
\label{thm:local}
Let $M=M_1 \# M_2\# \cdots \# M_n$ be the prime decomposition of a compact orientable $3$--manifold $M$ with $n\ge 2$.
Then $\pi_1(M)$ contains a generalized torsion element of order two if and only if so does $\pi_1(M_i)$ for some $i$.
\end{theorem}

So we may assume $M$ is prime. 
As we mentioned above, 
$S^2 \times S^1$ is the unique orientable $3$--manifold which is prime, 
but not irreducible. 
Since $\pi_1(S^2\times S^1)=\mathbb{Z}$ is abelian and torsion-free,  
it admits no generalized torsion elements.
Hence, in what follows we assume that $M$ is irreducible.

Obviously a usual torsion element is a generalized torsion element.
The next observation for finite $3$--manifold groups is well known (see \cite[p.452]{Sc1983}).

\begin{theorem}\label{thm:finite_case}
Let $G$ be the fundamental group of a closed orientable $3$--manifold.
Assume that $G$ is finite.
Then $G$ admits a torsion element of order two if and only if
$G$ is not a finite cyclic group of odd order.
\end{theorem}

Thus, for a finite $3$--manifold group $G$,
it always contains a generalized torsion element of order two, unless $G$ is a finite cyclic group
of odd order.
Conversely, for a finite cyclic group of odd order, a generalized torsion element is necessarily a torsion element,
so it cannot contain a generalized torsion element of order two.


In this article we will focus on generalized torsion elements of order two. 
As we observed, such a generalized torsion element appears in the fundamental group of the Klein bottle. 

Let $g$ be a generalized torsion element of order two in a torsion-free group $G$.
Then, by definition, there exist $a$ and $b$ in $G$ such that $g^ag^b=1$.
By taking a conjugation with $b^{-1}$, 
we have $g^{ab^{-1}}g =1$. 
Thus we may simplify the expression $g^ag^b=1$ to $g^c g =1$ 
using the non-trivial element $c = ab^{-1}$, equivalently $g g^{c} =1$. 
We call the pair $(g, c)$ a \textit{generalized torsion pair}; 
by notation, it is naturally understood that the generalized torsion element $g$ has order two. 
The simplified expression $g^c g = 1$ shows that $g^c = g^{-1}$, i.e. $g$ is conjugate to its inverse, and conversely,
such a non-trivial element gives a generalized torsion element of order two.

Assume that $(g, c)$ is a generalized torsion pair  in $G$. 
Let $H$ be a subgroup of $G$. 
If $g$ and $c$ belong to $H$, then we say that $(g,c)$ is a generalized torsion pair of $H$.
Even when $g \in H$, if $c \not\in H$, then $(g,c)$ is not a generalized torsion pair of $H$.

At first we state our classification theorem for torsion-free $3$--manifold groups admitting generalized torsion elements of order two, 
which generalizes \cite{HMT2023}. 

Throughout this article, we use the term Seifert fiber space to mean a $3$--manifold equipped with a fixed Seifert fibration. 
For a Seifert fiber space $X$, identifying each Seifert fiber to a point, we obtain a (possibly non-orientable) surface 
$B_X$, which is called a \textit{base surface}, and a natural projection $p_X \colon X \to B_X$. 
If $t$ is an exceptional fiber of index $p$, 
then $p_X(t) \in B_X$ is called a \textit{cone point} of order $p$.

Let $M$ be a  compact, orientable,  irreducible $3$--manifold. 
Consider the torus decomposition (JSJ--decomposition) \cite{JS1979,J1979} of $M$. 
A \textit{Seifert fibered decomposing piece} $X$ of $M$ is a decomposing piece with respect to the torus decomposition of $M$ which is a Seifert fiber space; 
possibly $X = \varnothing$ or $X = M$.

\begin{theorem}
\label{g-torsion_order2}
Let $M$ be a compact, orientable,  irreducible $3$--manifold whose fundamental group $G$ has no torsion. 
Assume that $G$ has a generalized torsion element $g$ of order two. 
Then $M$ has a Seifert fibered decomposing piece $X$ for which we have\textup{:} 
\begin{enumerate}
\item[(1)]
$g$ is conjugate to a generalized torsion element $g'$ of $\pi_1(X) \subset G$;
more precisely, 
$(g',c')$ is a generalized torsion pair of $\pi_1(X)$ for some $c'\in \pi_1(X)$,
and 
\item[(2)]
$X$ has an exceptional fiber of even index, 
or $B_X$ is non-orientable.

\end{enumerate}
\end{theorem}

In \cite{IMT_PAMS2019}, we investigate a behavior of generalized torsion elements under torus decomposition, 
and we demonstrate that a generalized torsion element  is not local in a sense that there are infinitely many toroidal $3$--manifolds, each of  whose fundamental group has a generalized torsion element, 
while the fundamental group of any decomposing piece has no such elements. 
However, restricting to a generalized torsion element of order two, 
Theorem~\ref{g-torsion_order2} (1) shows that such an element is local. 

\begin{corollary}
\label{local_g-torsion}
Any generalized torsion element of order two is local with respect to a torus decomposition. 
That is, if $g$ is a generalized torsion element of order two in the fundamental group of a toroidal $3$--manifold $M$, 
then it is conjugate to a generalized torsion element of order two in the fundamental group of a single decomposing piece of $M$. 
\end{corollary}

Theorem~\ref{g-torsion_order2} determines $3$--manifolds, each of whose fundamental group admits a generalized torsion element of order two. 
Now we proceed to classify generalized torsion elements of order two (up to conjugation)
\footnote{After the submission of this paper, Das and Das \cite{DD2024}
gave a similar classification of reversible elements, equivalently generalized torsion elements of order two, in the fundamental group of
a Seifert fiber space by using completely different arguments.
}.

Before stating our classification theorem, 
we prepare some elementary facts about generalized torsion pairs. 

For a generalized torsion pair $(g,c)$,
set $g'=g^x$ and $c'=c^x$ for $x \in G$.
Then the pair $(g',c')$ gives a generalized torsion pair.
For, the equation $g^cg=1$ implies 
\[
1=(g^cg)^x= (g^c)^xg^x=(g^x)^{c^x}g^x=(g')^{c'} g'.
\]

\begin{lemma}
\label{property_g-pair}
 Let $(g,c)$ be a generalized torsion pair in a group $G$, and
let $d$ be an element such that $c^{-1}d$ belongs to the centralizer of $g$. 
Then $(g, d)$ is also a generalized torsion pair.
\end{lemma}

\begin{proof}
Since $c^{-1}d$ lies in the centralizer of $g$, $(c^{-1}d) g = g (c^{-1}d)$, 
which implies $g^{c^{-1}}=g^{d^{-1}}$.
Since $(g, c)$ is a generalized torsion pair,  
we have $g^c g = 1$, which also implies $g g^{c^{-1}} = 1$. 
Thus $g g^{d^{-1}} = 1$, i.e., $g^d g = 1$, and $(g, d)$ is also a generalized torsion pair. 
\end{proof}

\begin{lemma}
\label{power}
Assume that $G$ is torsion-free. 
If $(g,c)$ is a generalized torsion pair of $G$, 
then so is $(g^n,c)$ for any integer $n\ne 0$.
\end{lemma}

\begin{proof}
Since $(g, c)$ is a generalized torsion pair, 
we have $g^c g =1$,
 so $g^c=g^{-1}$.  Then $g^n\ne 1$ (by the assumption) and
\[
(g^n)^c g^n = (g^{c})^n g^n =(g^{-1})^n g^n= g^{-n}g^n=1. 
\]
Hence $(g^n, c)$ is also a generalized torsion pair. 
\end{proof}

We say that $(g,c)$ is \textit{primitive\/} if
there is no generalized torsion pair $(h,c)$ with $g=h^n$ for some $n\ne \pm 1$.
Following Lemma~\ref{power} we restrict our attention to primitive generalized torsion pairs $(g, c)$.  

\begin{remark}
\label{order2_power}
Lemma~\ref{power} shows that any non-trivial power of a generalized torsion element of order two is also 
a generalized torsion element of order two in a torsion-free group. 
This is a distinguished property of generalized torsion elements of order two. 
In general we do not expect that a power of a generalized torsion element is a generalized torsion element again. 
\end{remark}

Let $G$ be the fundamental group of irreducible, aspherical $3$--manifold $M$.
Then $G=\pi_1(M)$ is torsion-free. 
Following Theorem~\ref{g-torsion_order2}, 
if $G$ has a generalized torsion element $g$ of order two, 
then $M$ has a Seifert fibered decomposing piece $X$, which has an exceptional fiber of even index or non-orientable base surface $B_X$. 
Taking a suitable conjugation we have a generalized torsion pair $(g, c)$ of $\pi_1(X)$. 
The next result classifies generalized torsion pairs of $\pi_1(X)$.

When $B_X$ has a cone point, we choose a cone point as the base point $x_0$.
Otherwise, $x_0$ is any point in $B_X$. 
We say that a loop with a base point $x_0$ on $B_X$ is \textit{non-trivial} if it is not homotoped to the base point $x_0$ without passing other cone points of $B_X$. 

To state our result we need some terminologies. 
\begin{definition}
\label{folded_loop}
Let $x_0$ is the base point, and $x_1$ another point on $B_X$.
\begin{enumerate}
\item
A closed curve $\alpha \colon [0, 1] \to B_X$ is  an
\textit{$x_0$--folded loop based at $x_0$} if 
$\alpha(t) = \alpha(1-t)$, 
$\alpha(0) = \alpha(1) = x_0$ and $\alpha(1/2) = x_0$; see Figure~\ref{folded_curve} (Left). 

\item
A closed curve $\alpha \colon [0, 1] \to B_X$ is an
\textit{$x_1$--folded loop based at $x_0$} if $\alpha(t) = \alpha(1-t)$, 
$\alpha(0) = \alpha(1) = x_0$ and $\alpha(1/2) = x_1$; see Figure~\ref{folded_curve} (Right). 

\item
We say that an $x_0$--folded loop $\alpha$ based at $x_0$ is \textit{non-trivial} if $\alpha$ cannot be homotoped to a constant map 
$e_{x_0}$ keeping $\alpha(0), \alpha(1/2)$ and $\alpha(1)$ fixed and does not pass any cone point of $B_X$. 

\end{enumerate}
\end{definition}

In the following, for notational simplicity, 
a representative closed curve of $g \in \pi_1(X)$ is also denoted by the same symbol $g$, 
and apply the above terminologies to $p_X^*(g)=p_X \circ g \colon [0, 1] \to B_X$ for $g \in \pi_1(X)$.

\begin{theorem}
\label{classification}
Let $X$ be an irreducible Seifert fiber space with infinite fundamental group.
Let  $(g, c)$ be a primitive generalized torsion pair in $\pi_1(X)$.
Then $(g, c)$ satisfies one of the following. 
\begin{enumerate}
\renewcommand{\labelenumi}{(\roman{enumi})}
\item
$c$ is a power of an exceptional fiber $t_0$ of even index with $x_0 = p_X(t_0)$, 
and $p_X^*(g)$ is a non-trivial, orientation preserving $x_0$--folded loop based at $x_0$ on $B_X$. 
See Figure~\ref{folded_curve} \textup{(Left)}.
\item
$c$ is a power of an exceptional fiber $t_0$ of even index with $x_0 = p_X(t_0)$, 
and $p_X^*(g)$ is an $x_1$--folded loop based at $x_0$ on $B_X$, 
where $x_1$ is a cone point $p_X(t_1)$ for some exceptional fiber $t_1 \ne t_0$ of even index. 
See Figure~\ref{folded_curve}\textup{ (Right)}.
\item 
$g$ is a regular fiber $t_0$, 
and $p_X^*(c)$ is an orientation reversing loop based at $x_0 = p_X(t_0)$ in $B_X$ \textup{(}hence $B_X$ is non-orientable\textup{)}.
\item
$g$ is an orientation preserving loop and $c$ is an orientation reversing loop of a \textup{(}possibly immersed\textup{)} horizontal Klein bottle,
 when 
$B_X$ is the projective plane $\mathbb{R}P^2$ or $X$ is a circle bundle over the Klein bottle. 
\end{enumerate}
\end{theorem}

\begin{figure}[h]
\begin{center}
\includegraphics*[bb=0 0 483 281, width=0.7\textwidth]{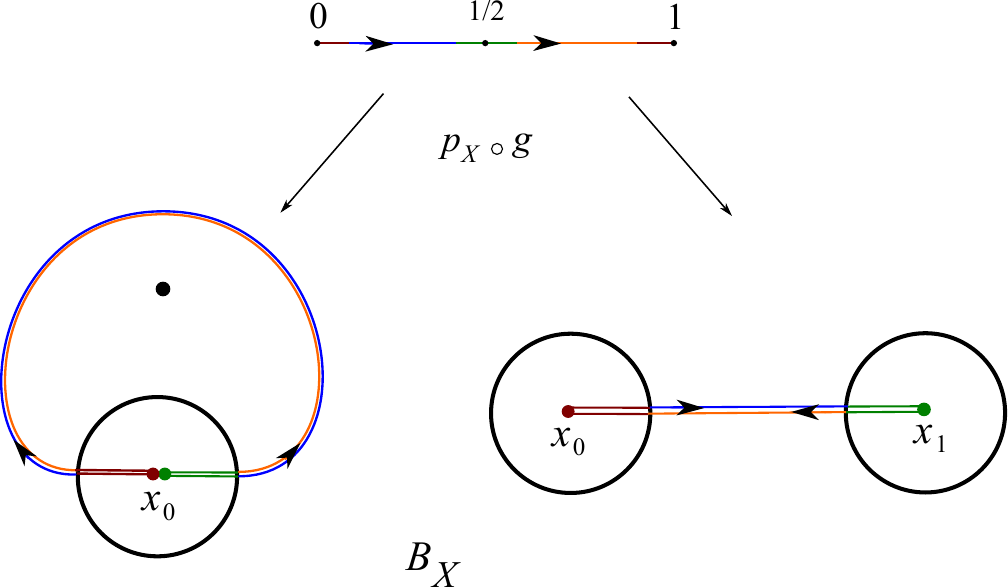}
\end{center}
\caption{A loop whose base point is a single cone point of even order and an arc whose end points are distinct cone points of even order.}
\label{folded_curve}
\end{figure}

\begin{remark}
\label{same_projection}
As shown in Lemma~\ref{property_g-pair}, 
a generalized torsion pair $(g, c)$ provides infinitely many such pairs using elements of the centralizer of $g$. 
These generalized torsion pairs have the same description as $(g,c)$ in Theorem~\ref{classification}.
\end{remark}

The following result generalizes \cite[Theorem 1.4]{HMT2023} to $3$--manifold groups. 
Recall that a group $G$ is an \textit{$R$--group\/} if $x^n=y^n$ for $n\ne 0$ implies $x=y$.
In particular, an $R$--group is torsion-free.
Also, $G$ is an  \textit{$\overline{R}$--group\/} if $G$ is torsion-free and the centralizer and the normalizer of the isolator subset $I\langle x\rangle$ for 
the cyclic subgroup $\langle x \rangle$ coincide for any $x\in G$.

In our context, it is convenient to adopt an alternative definition of $\overline{R}$--group (rather than the above original one) as in Lemma~\ref{lem:Rbar} in terms of a Baumslag--Solitar equation \cite{FW1999}. 
Such an equation in a $3$--manifold group is completely understood by a strong result by Shalen \cite{Sha2001}. 
In \cite{HMT2023} (see Lemma~\ref{lem:R-gpair}),
we observe that any $R$--group does not admit a generalized torsion element of order two,
which enables us to prove: 

\begin{theorem}
\label{RRbar}
Let $M$ be a compact orientable $3$--manifold.  Let $G=\pi_1(M)$.
Then $G$ is an $\overline{R}$--group
if and only if $G$ is an $R$--group.
\end{theorem}

A characterization of Haken $3$--manifold groups which are $R$--groups is given by
Johannson \cite[Proposition 32.4]{J1979}, 
where several families, such as Seifert fiber spaces over the sphere with three exceptional fibers, torus bundles over the circle,
are excluded from the consideration. 
Although its generalization to $3$--manifold groups may be known for experts, 
for completeness we establish the following characterization of $3$--manifold groups which are $R$, and hence $\overline{R}$--groups (Theorem~\ref{RRbar}).

We recall that the fundamental group of a compact orientable $3$--manifold $M$ is $R$--group if and only
if each prime factor $M_i$ of $M$ has such fundamental group; see \cite[Proposition~2.8]{FW1999}.  
So we may assume that $M$ is prime.

Since $\pi_1(M)$ has a torsion element only if 
$M$ is either non-prime or spherical, i.e., $\pi_1(M)$ is finite (see \cite[Lemma 9.4]{Hem1976}). 
Thus, if $\pi_1(M)$ has a torsion element, then $M$ is spherical. 
Obviously $\pi_1(M)$ is never an $R$--group except when $M = S^3$. 
These observation leads us to restrict our attention to compact orientable prime $3$--manifolds with $|\pi_1(M)| = \infty$.

\begin{theorem}
\label{R-group_characterization}
Let $M$ be a compact orientable,  prime $3$--manifold with $G = \pi_1(M)$ infinite. 
Suppose that $M$ is neither  $S^2\times S^1$ nor a solid torus. 
Then $G$ is an $R$--group if and only if $M$ does not have a Seifert fibered decomposing piece $X$ which has either \textup{(}i\textup{)} non-orientable base surface or 
 \textup{(}ii\textup{)} an exceptional fiber.
\end{theorem}

\begin{remark}
In the statement of Theorem \ref{R-group_characterization}, we exclude $S^2\times S^1$ and a solid torus.
Their fundamental groups are $\mathbb{Z}$, which is an $R$--group.
However, they admit Seifert fibrations with exceptional fibers.
\end{remark}

\section{Geometry of generalized torsion elements of order two}
\label{sec:g-torsion}

One of our interest is understanding geometric meaning of generalized torsion elements in $3$--manifold groups.  
Restrict our attention to such elements of order two, we may develop this viewpoint. 

Let $F$ be a closed surface, which is  possibly non-orientable. 
We say that a smooth map $f \colon F \to M$ is \textit{$\pi_1$-injective\/} if the induced homomorphism on the fundamental groups is a monomorphism. 
In the following, we abuse notation by referring to the image $f(F) \subset M$ as $F$.

The next lemma claims that if the $3$--manifold group admits a generalized torsion element of order two, 
then there exists a $\pi_1$-injective map from the Klein bottle into the $3$--manifold. 
(It is a generalization of \cite[Lemma 4]{HMT2023}.)
This is the starting point which gives a geometric interpretation of generalized torsion element of order two.

\begin{lemma}
\label{Kbmap}
Let $M$ be a $3$--manifold with $G=\pi_1(M)$ torsion-free.
If $G$ admits a generalized torsion element $g$ of order two,
then there exists a $\pi_1$-injective map $f \colon F\to M$ from the Klein bottle $F$ 
with $\pi_1(F) =  \langle a,b \mid a^{-1}ba=b^{-1}\rangle$ such that 
$f(b) = g$. 
\end{lemma}

\begin{proof}
Let $g$ be a generalized torsion element of order two in $G$.
Then there exists $c$ such that  $c^{-1}gc = g^{-1}$. 
Since $g \ne 1$ and $G$ is torsion-free, 
we have $c\ne 1$.
We use the same symbols $g$ and $c$ to denote the loops with the base point $p_0$.

For the Klein bottle $F$, take two loops $a$ and $b$ meeting in a single point $q_0$ so that
$a$ is orientation-reversing but $b$ orientation-preserving.
Then they give a presentation $\pi_1(F)=\langle a,b\mid a^{-1}ba=b^{-1}\rangle$ based on the point $q_0$.
Let $f$ be a map sending $q_0$, $a$ and $b$ to  $p_0$, $c$ and $g$, respectively.
Since the image  $c^{-1}g c g$ of the loop $a^{-1}bab$ is null-homotopic in $M$,
$f$ extends to a map on $F$.

We claim that
the induced homomorphism $f_*\colon \pi_1(F)\to G$ is injective.

Since we may deduce $b^{\varepsilon} a^{\delta} = a^{\delta} b^{-\varepsilon}$ ($\varepsilon, \delta = \pm 1$) 
from the relation $a^{-1} b a = b^{-1}$, 
any element of $\pi_1(F)$ is written as $a^ib^j$ for some integers $i$ and $j$.
Assume that $f_*$ is not injective.
Then there exists a non-trivial element $a^ib^j$ such that $f_*(a^ib^j)=c^i g^j=1$.
Since $g$ and $c$ are not torsions, $i\ne 0$ and $j\ne 0$.

On the other hand, the relation $c^{-1} g c= g^{-1}$ gives
$c^{-1} g^j c= g^{-j}$.
Since $g^j=c^{-i}$, we have 
$g^{-j} = c^{-1} g^j c = c^{-1} c^{-i} c = c^{-i} = g^j$, 
so $g^{2j}=1$.
This is impossible.
\end{proof}

For a Haken $3$--manifold  with incompressible boundary,
there exists the characteristic submanifold  $V$ by Jaco--Shalen \cite{JS1979} and Johannson \cite{J1979}.
Then $V$ is a disjoint union of Seifert fiber spaces\footnote{
There is a slight difference between two theories of \cite{JS1979} and \cite{J1979}.
For a hyperbolic $3$--manifold $M$ with non-empty toroidal boundary, 
$V$ is empty in \cite{JS1979}, but  $V=\partial M \times I$ in \cite{J1979}.
We adopt the latter.
}.

\begin{lemma}
\label{torsion2}
Let $M$ be an irreducible $3$--manifold and 
$f \colon F \to M$ a $\pi_1$-injective map from the Klein bottle $F$ 
such that $f(b) = g$ and $f_* \colon \pi_1(F) \to G=\pi_1(M)$ is injective given
 as in Lemma~\ref{Kbmap}.
Then there exists a Seifert fibered decomposing piece $X$ such that $f$ may be homotoped to a $\pi_1$-injective map $h$ with $h(F) \subset X$. 
\end{lemma}

\begin{proof}
Assume first that $M$ is sufficiently large. 
Let $N$ be a twisted $[-1, 1]$--bundle over the Klein bottle $F$ in which $F$ is the zero section. 
Denote the natural retraction $N \to F$ by $r$. 
Note that $r$ induces an isomorphism $r_* \colon \pi_1(N) \to \pi_1(F)$. 
Then $f \circ r \colon N \to M$ induces a monomorphism $(f \circ r)_* \colon \pi_1(N) \to G$, and thus $f \circ r$ is non-degenerate. 
Then \cite[III.5.5. Proposition]{JS1979} asserts that there exists a Seifert fibered decomposing piece $X \subset M$, 
and $f \circ r$ is homotopic to a $\pi_1$-injective map $f' \colon N \to M$ such that $f'(N) \subset X$. 
This then means that $f \colon F \to M$ is homotoped to $h = f'|_{F} \colon F \to M$ such that $h(F) \subset X$. 

Suppose that $M$ is not sufficiently large. 
Then $M$ is closed \cite{Hem1976,Ja1980}, 
and since $M$ is irreducible and $G$ is infinite, the Geometrization implies that $M$ is either a Seifert fiber space or a hyperbolic $3$--manifold.  
In the former case, we have a desired result, so
we assume $M$ is a closed hyperbolic $3$--manifold. 
Note that $F$ is double covered by a torus, and $\pi_1(F)$ has a subgroup isomorphic to $\Gamma = \mathbb{Z} \oplus \mathbb{Z}$. 
Then since  $f_* \colon \pi_1(F) \to G$ is injective, $G$ has a subgroup $f_*(\Gamma) \cong \mathbb{Z} \oplus \mathbb{Z}$ whose holonomy image in 
$\mathrm{Isom}^+(\mathbb{H}^3)$ generated by parabolic elements (see \cite[Theorem 1.7.2]{AFW2015} or \cite[Corollary 4.6]{Sc1983}). 
Hence $M$ should have torus cusp, which contradicts $M$ being closed. 
\end{proof}

\begin{remark}
We may assume that the $\pi_1$-injective map $h \colon F \to X$ is smooth. 
Furthermore, 
$h \colon F \to X$ may be homotoped to an immersion. 
See \cite[Lemma~1.2]{HH1985}. 
For notational simplicity we denote this immersion by the same symbol $h \colon F \to X$. 
\end{remark}

We apply a result  due to Hass \cite{Hass1984} to the Klein bottle to obtain: 

\begin{proposition}
\label{verical_horizontal}
Let $X$ be a Seifert fiber space with possibly non-empty boundary.\footnote{We appreciate Joel Hass for telling us his result \cite{Hass1984} 
still holds for Seifert fiber spaces with non-empty boundary.}
Any $\pi_1$-injective immersion $h \colon F \to X$ from the Klein bottle $F$ may be homotoped to a vertical or horizontal immersion
$\varphi \colon F \to X$. 
\end{proposition}

Now let us investigate a vertical or horizontal immersion $\varphi \colon F \to X$ in details. 
In the following we assume that $X$ has exceptional fibers $t_1, \dots, t_n$ of indices $\alpha_1, \dots, \alpha_n$. 
(When it has no exceptional fibers, $n$ is understood to be $0$.)

\begin{lemma}
\label{vertical}
If $\varphi \colon F \to X$ is vertical, 
then $B_X$ has a cone point of even order, or $B_X$ is non-orientable. 
\end{lemma}

\begin{proof}
Since $\varphi \colon F \to X$ is vertical, 
we have a foliation on the Klein bottle $F$ induced from the Seifert fibration of $X$. 
Clearly, each leaf of such a foliation is a circle.

Recall that there are just four isotopy classes of non-trivial simple closed curves on the Klein bottle $F$ (\cite{Price1977,Rubin1979}). 
Presenting $\pi_1(F) = \langle a, b \mid aba^{-1} b = 1 \rangle$, 
then these are represented by $a,\ b,\ ab,\ a^2$.  
See Figure~\ref{Klein}. 
Note that $a$ and $ab$ are represented by orientation reversing curves in $F$, 
while $b$ and $a^2$ are represented by orientation preserving curves in $F$.  

These simple closed curves define a foliation $\mathcal{F}$ of $F$. 
If $a$ is a leaf of $\mathcal{F}$, then cutting along $a$, we obtain a M\"obius band whose core is a leaf $a'$, and  
orientation preserving curves isotopic to $a^2$ give the other leaves of the M\"obius band. 
This gives a foliation $\mathcal{F}_a$ in Figure~\ref{Klein} (i). 

If $b$ is a leaf, then cutting along $b$, we obtain an annulus which is foliated by simple closed curves isotopic to $b$. 
This gives a a foliation $\mathcal{F}_b$ in Figure~\ref{Klein} (ii). 

If $ab$ is a leaf, then we apply a cut and paste operation.
We obtain a foliated Klein bottle which is essentially the same as $\mathcal{F}_a$; 
see Figure~\ref{Klein}(iii). 

Thus we have two distinct foliations $\mathcal{F}_a$ and $\mathcal{F}_b$ on $F$. 

\begin{figure}[h]
\begin{center}
\includegraphics*[bb=0 0 845 715, width=0.8\textwidth]{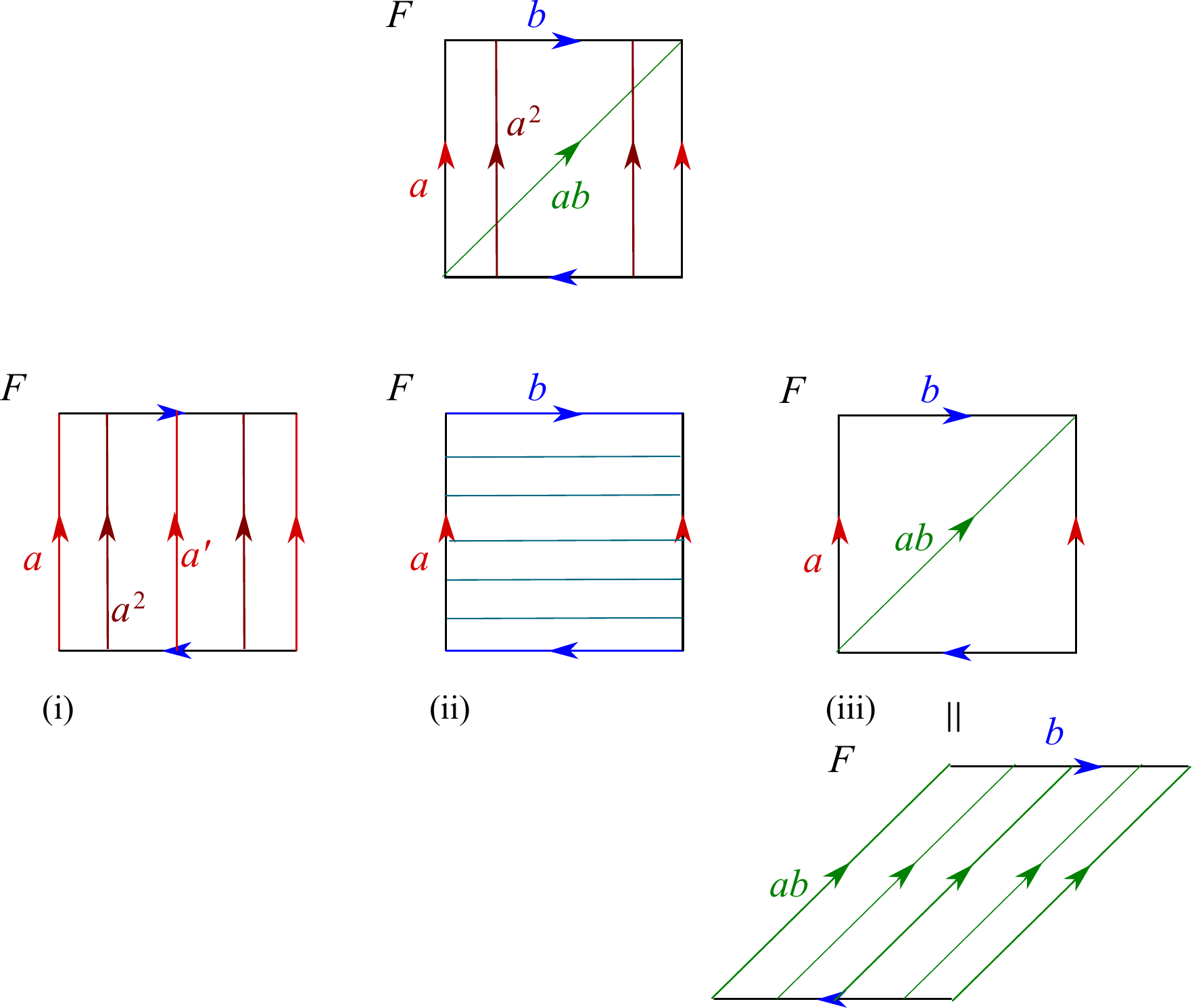}
\end{center}
\caption{Foliations on the Klein bottle induced from a Seifert fibration.}
\label{Klein}
\end{figure}

To prove the lemma, 
it is sufficient to show that if $B_X$ is orientable, 
then it has a cone point of even order. 

Assume that $B_X$ is orientable. 
Then an orientation reversing loop of the Klein bottle $F$ should be a leaf of this induced foliation, 
and hence the induced foliation is $\mathcal{F}_a$. 
Note that we have two non-isotopic orientation reversing loops $a$ and $a'$ on the Klein bottle $F$ (Figure~\ref{Klein}(i)).  
Consider small tubular neighborhoods $N(a)$ and $N(a')$ of $a$ and $a'$ in $F$; 
these are M\"obius bands. 
See Figures~\ref{immersed_Klein_m=1} (Upper Left) and \ref{immersed_Klein_m=3} (Upper Left). 
$\partial N(a)$ is a single loop which is also a leaf of the foliation of $F$, and thus $\varphi(\partial N(a))$ is a fiber in the Seifert fibration of $X$. 
Let us write the image $\varphi(a) = t_0$ and $\varphi(a') = t_1$, where $t_1$ may coincide with $t_0$. 
Figures~\ref{immersed_Klein_m=1} (Middle) and \ref{immersed_Klein_m=3} (Middle) describe a situation 
with $t_0 = t_1$, 
and Figures~\ref{immersed_Klein_m=1} (Right) and \ref{immersed_Klein_m=3} (Right) describe a situation 
in which $t_0 \ne t_1$. 
It should be noted here that $\varphi |_{a} \colon a \to t_0$ is not necessarily a homeomorphism. 
In general, 
it is an $m$--fold covering for some $m\ge 1$, 
and $\varphi(\partial N(a))$ wraps $2m$ times along $t_0$. 
This means that $t_0 = \varphi(a)$ is an exceptional fiber of even index $2m$. 
Similarly, $\varphi |_{a'} \colon a' \to t_1$ is an $n$--fold covering for some $n\ge 1$, 
and $\varphi(\partial N(a'))$ wraps $2n$ times along $t_1$. 
Hence 
$t_1 = \varphi(a')$ is an exceptional fiber of even index $2n$. 

In the case where $t_0 \ne t_1$, covering degrees $m$ and $n$ may be distinct, 
and so correspondingly the indices of exceptional fibers $t_0$ and $t_1$ may be distinct. 

Figure~\ref{immersed_Klein_m=1} illustrates a case where $m = 1$, 
and Figure~\ref{immersed_Klein_m=3} (Right)  illustrates a case where $m = 3, n = 2$. 

On the other hand, 
$F - \mathrm{int}(N(a) \cup N(a'))$ is an annulus $A$, 
and $\varphi(A)$ is a possibly immersed vertical annulus in $X$ such that 
$\partial \varphi(A) = \partial \varphi(N(a)) \cup \partial \varphi(N(a'))$. 

\begin{figure}[h]
\begin{center}
\includegraphics*[bb=0 0 1018  523, width=1.0\textwidth]{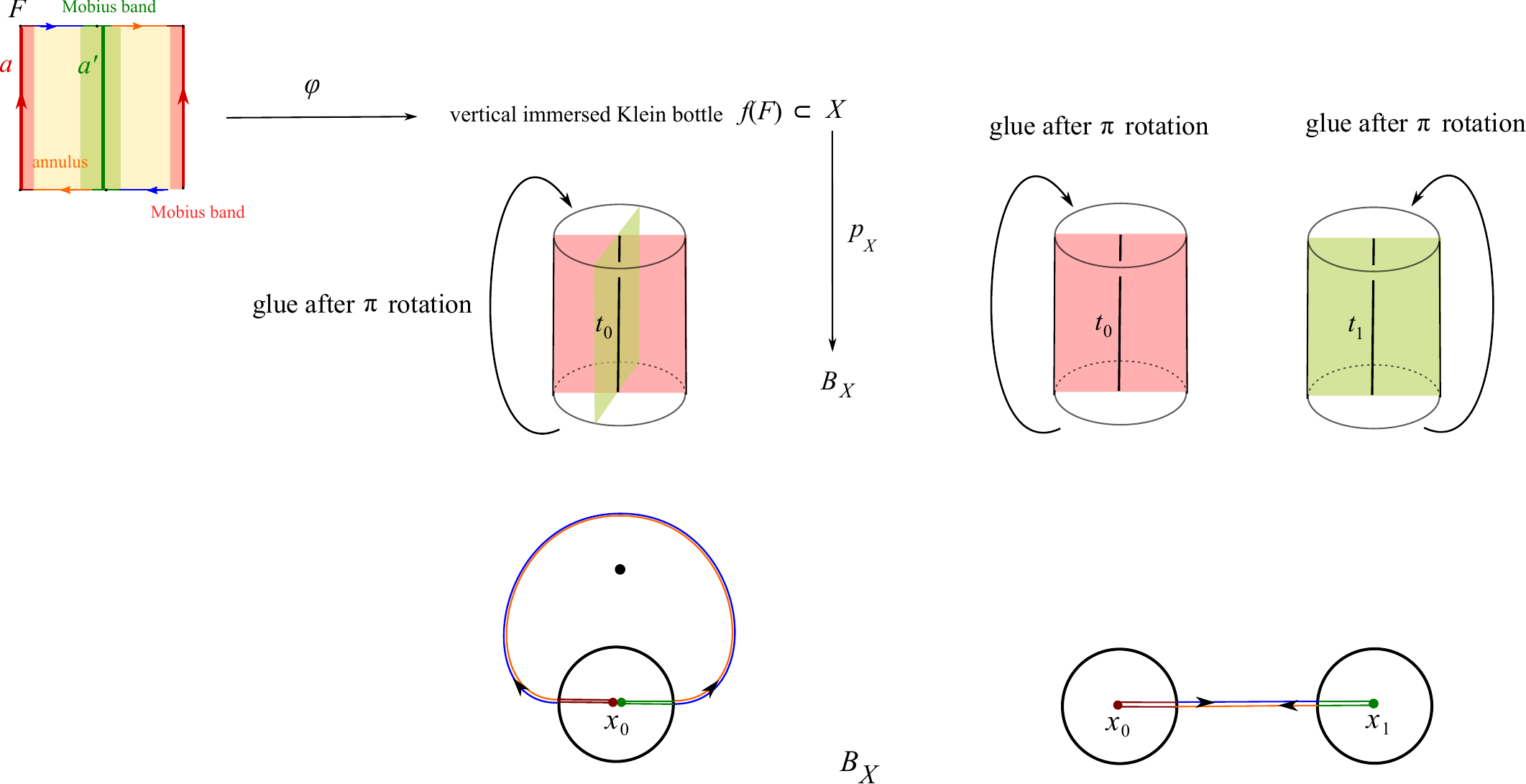}
\end{center}
\caption{Immersed Klein bottles and generalized torsion elements of order two.}
\label{immersed_Klein_m=1}
\end{figure}

\begin{figure}[h]
\begin{center}
\includegraphics*[bb=0 0 1001 522, width=1.0\textwidth]{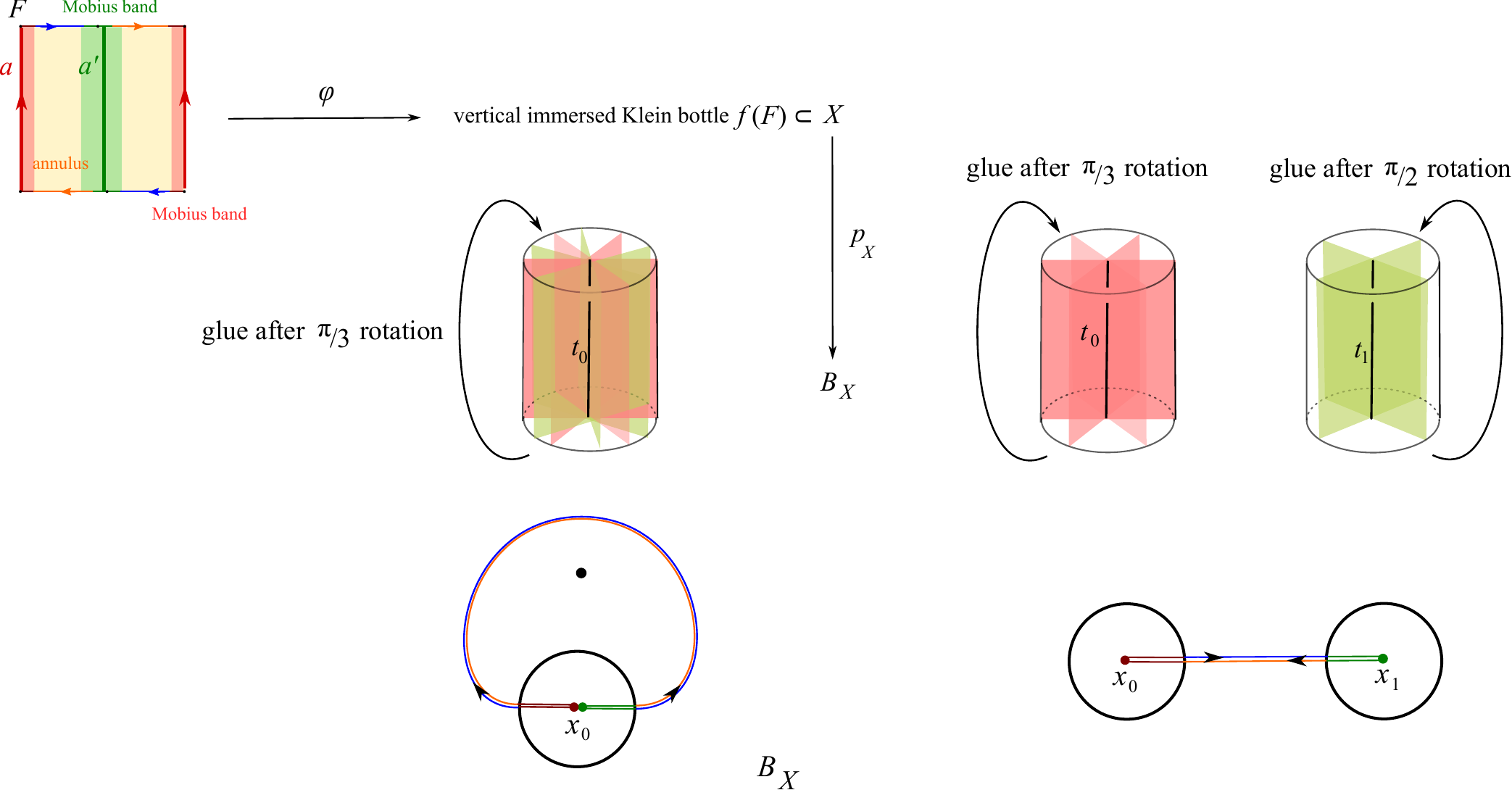}
\end{center}
\caption{Immersed Klein bottles and generalized torsion elements of order two.}
\label{immersed_Klein_m=3}
\end{figure}

Hence the Seifert fiber space $X$ has an exceptional fiber of even index, 
in other words, $B_X$ has a cone point of even order  as desired. 
\end{proof}

\begin{lemma}
\label{horizontal}
If $\varphi \colon F \to X$ is horizontal, 
then $B_X$ is either the projective plane $\mathbb{R}P^2$ with two cone points of order two, 
or the Klein bottle without cone points. 
\end{lemma}

\begin{proof}
Take a composition $p_X \circ \varphi \colon F \to B_X$, 
which is a branched cover. 
Since $F$ is non-orientable and $\partial F = \varnothing$, 
$B_X$ is also non-orientable and $\partial B_X = \varnothing$. 
The Riemann--Hurwitz formula implies that 
\[
0 = \chi(F) = p\left( \chi(B_X) - \sum_{i = 1}^n \left(1 - \frac{1}{\alpha_i}\right)\right),
\]
where $p$ is the covering degree and $\alpha_1,\dots,\alpha_n$ are the orders of cone points.
Then we have
\[
\chi(B_X)=\sum_{i=1}^n \left( 1-\frac{1}{\alpha_i}\right).
\]

If $n=0$, then $\chi(B_X)=0$, so $B_X$ is the Klein bottle.

If $n=1$, then the integer $\chi(B_X)$ satisfies 
\[
0<\chi(B_X)=1-\frac{1}{\alpha_1} < 1,
\]
which is impossible.

If $n=2$, then 
\[
1 \le \chi(B_X) = 2-\left( \frac{1}{\alpha_1}+\frac{1}{\alpha_2}\right) < 2,
\]
so $\chi(B_X)=1$ and $\alpha_1=\alpha_2=2$. 
Hence $B_X$ is $\mathbb{R}P^2$ with two cone points of order two. 

Finally, if $n>2$, then
\[
\chi(B_X) = n-\sum_{i=1}^n\frac{1}{\alpha_i}\ge n-\frac{n}{2}=\frac{n}{2}>1,
\]
which is impossible.
\end{proof}

\section{Generalized torsion elements of order two in $3$--manifold groups}

\subsection{Classification of $3$--manifolds with generalized torsion elements of order two.}\quad

The aim of this section is to classify $3$--manifolds, each of  whose fundamental group has a generalized torsion element of order two (Theorems  \ref{thm:local} and \ref{g-torsion_order2}), 
and then to establish a classification of such generalized torsion pairs (Theorem~\ref{classification}).

We start with a proof of Theorem \ref{thm:local}. 
In the proof, we use stable commutator length. 
So we begin by recalling its definition and collecting some properties which we need in the proof.

For $g \in [G,G]$, the \emph{commutator length\/} $\mathrm{cl}_G(g)$ is the smallest number of commutators in $G$ whose product is equal to $g$. 
The \emph{stable commutator length\/} $\mathrm{scl}_{G}(g)$ of $g \in [G,G]$ is defined to be the limit
\begin{equation}
\label{def1}
\mathrm{scl}_{G}(g) = \lim_{n \to \infty} \frac{\mathrm{cl}(g^n)}{n}.
\end{equation}

Since $\mathrm{cl}_G(g^n)$ is non-negative and subadditive, Fekete's subadditivity lemma shows that 
this limit exists.  

We will extend (\ref{def1}) to the stable commutator length $\mathrm{scl}(g)$ for an element $g$ which is not necessarily in $[G, G]$ as 
\begin{equation*}
\label{def2}
\mathrm{scl}_G(g) = 
\begin{cases}
\frac{\mathrm{scl}_G(g^k)}{k} & \text{if}\, g^k \in [G,G]\, \text{for some}\, k>0, \\
\infty & \text{otherwise}.
\end{cases}
\end{equation*}
It is well known that $\mathrm{scl}_G(g)$ is independent of the choice of $k>0$ such that $g^k \in [G,G]$, for instance, see \cite{IMT_PAMS2019}.

\begin{proof}[Proof of Theorem \ref{thm:local}]
The ``if'' part is obvious, because each $\pi_1(M_i)$ is a subgroup of $\pi_1(M)$.
Hence we assume that $\pi_1(M)$ contains a generalized torsion element $g$ of order two.

First, assume that $g$ is conjugate into $\pi_1(M_i)$ for some $i$.
Hence there exists $a \in \pi_1(M)$ such that $a^{-1}ga=g_1\in \pi_1(M_i)$.
We can prove that $g_1$ is a generalized torsion element of order two as an element of $\pi_1(M_i)$
by the same procedure as in the proof of Theorem 1.5 in \cite{IMT_PAMS2019}.

Thus we can assume that $g$ is not conjugate into any factor. 
In the following we show that there exists a torsion element of order two in $\pi_1(M_i)$ for some $i$. 

Let us express $g$ uniquely as a reduced form $g_1g_2\dots g_k$ where $k\ge 2$, $g_i\in \pi_1(M_{j_i})$, $g_i\ne 1$, and $j_i\ne j_{i+1}$.
For simplicity, we assume that $j_1=1$, and express $\pi_1(M)$ as $A*B$ where $A=\pi_1(M_1)$ and $B=\pi_1(M_2)*\cdots * \pi_1(M_n)$.
Then $g=a_1b_1\dots a_Lb_L$ with $a_i\in A-\{1\}$, $b_i\in B-\{1\}$ and $L\ge 1$.

Recall that $g^2\in [\pi_1(M),\pi_1(M)]$, because $gg^c=1$ in $\pi_1(M)$ implies $2[g]=0$ under the abelianization.
Then \cite[Theorem 3.1]{Ch2018} implies 
\[
\mathrm{scl}(g^2)\ge \frac{1}{2}-\frac{1}{N},
\]
where $N\ (\ge 2)$ is the minimal order of $a_i$ and $b_i$.

Since, by definition, 
$\mathrm{scl}(g)=\mathrm{scl}(g^2)/2$, 
and $\mathrm{scl}(g)=0$ (see \cite[Theorem~2.4]{IMT_PAMS2019}), we have $N=2$. 
If some $a_i \in \pi_1(M_1)$ is a torsion element of order two, 
then it is a desired torsion element. 
Assume that some $b_i \in \pi_1(M_2) * \cdots * \pi_1(M_n)$ is a torsion element of order two. 
Then apply \cite[Corollary~4.1.4]{MKS2004} repeatedly to find a torsion element of order two in $\pi_1(M_j)$ for some $j = 2, \dots, n$, 
which is the required torsion element. 

This completes a proof. 
\end{proof}

\begin{proof}[Proof of Theorem~\ref{g-torsion_order2}]
Let $M$ be an orientable,  irreducible
$3$--manifold with $G = \pi_1(M)$ torsion-free.
Assume that $G$ has a generalized torsion element $g$ of order two. 
Then by Lemma~\ref{Kbmap} we have a smooth $\pi_1$-injective map $f \colon F \to M$ from the Klein bottle $F$. 
Apply Propositions~\ref{torsion2} and \ref{verical_horizontal} to see that $f$ may be homotoped to a $\pi_1$-injective immersion 
$\varphi \colon F \to M$ 
so that $\varphi(F) \subset X$ for some Seifert fibered decomposing piece $X$ (with respect to the torus decomposition of $M$), 
for which  $\varphi \colon F \to X$ is either vertical or horizontal. 
($X$ may be $M$ itself if $M$ is Seifert fibered.)

Thus we have the following, which is the first assertion of Theorem~\ref{g-torsion_order2}. 
See Lemma \ref{Kbmap} for the choice of $a$ and $b$. 

\begin{claim}
\label{local}
Let $g$ be a generalized torsion element and $c$ a non-trivial element such that $g^c g = 1$. 
Then they are simultaneously conjugate to $\varphi(b)$ and $\varphi(a)$ in $\pi_1(X)$, respectively. 
In particular, 
$(\varphi(b), \varphi(a))$ is a generalized torsion pair of $\pi_1(X)$.
\end{claim}

Now we characterize the Seifert fibered decomposing piece $X$.

\noindent
\textbf{Case 1}. $\varphi \colon F \to X$ is vertical. 
Lemma~\ref{vertical} asserts that $B_X$ has a cone points of even order, or $B_X$ is non-orientable.

\noindent
\textbf{Case 2}. $\varphi \colon F \to X$ is horizontal. 
Following Lemma~\ref{horizontal}, we see that $B_X$ is either $\mathbb{R}P^2$ with two cone points of order two, or the Klein bottle without cone points. 
In either case $B_X$ is non-orientable.

This completes a proof of Theorem~\ref{g-torsion_order2}. 
\end{proof}

\subsection{Classification of generalized torsion pairs}

In this subsection we will give a classification of generalized torsion pairs $(g, c)$, 
which proves Theorem~\ref{classification}.

Let $F$ be the Klein bottle with a fixed presentation
\[
\pi_1(F) = \langle a, b \mid \ a^{-1} b a b = 1 \rangle.
\]

For later convenience, we first classify generalized torsion pairs of the fundamental group of the Klein bottle $F$. 

\begin{lemma}
Any generalized torsion pair of $\pi_1(F)$ can be expressed as $(b^m,a^{2i+1}b^j)$  for some integers $m\ne 0$, $i$ and $j$. 
In particular, only non-zero powers of $b$ are generalized torsion elements of order two.
\end{lemma}

\begin{proof}
Let $g$ be a generalized torsion element of order two.
As in the proof of Lemma \ref{Kbmap},
$g$ is uniquely expressed as $g=a^ib^j$.

We have $g^cg=1$ for some $c$.
Hence $2[g]=0$ in $H_1(F)$.
Since $H_1(F)=\mathbb{Z}\oplus \mathbb{Z}_2$, where each factor is generated by $[a]$ and $[b]$, respectively,
$[g]=m[b]$ for some integer $m$.
Thus $i=0$ and $j=m$, so $g=b^m$.
Since $g\ne 1$, we have $m \ne 0$.

The relation $a^{-1} b a b = 1$ shows that $(b, a)$ is a generalized torsion pair.
By Lemma \ref{order2_power}, $(b^m,a)$ is also a generalized torsion pair. 
Furthermore, by Lemma~\ref{property_g-pair} $(b, c)$ (and hence $(b^m, c)$) is a generalized torsion pair as well 
if $a^{-1}c$ belongs to the centralizer of $b$. 

\begin{claim}
\label{centralizer}
The centralizer of $b$ consists of elements $a^{2i} b^j$ for $i, j\in \mathbb{Z}$. 
\end{claim}

\begin{proof}
Since $a^2 b = a (ab) = a (b^{-1} a) = (ab^{-1}) a  = (b a) a = b a^2$, $a^2$ belongs to the centralizer of $b$. 
Thus $a^{2i} b^j$ belongs to the centralizer of $b$. 
Assume for a contradiction that $a^{2i+ 1} b^j$ belongs to the centralizer of $b$. 
Then $b(a^{2i+ 1} b^j) = (a^{2i+ 1} b^j) b$, which implies that 
$ba (a^{2i} b^j) = a (a^{2i} b^j) b = ab (a^{2i} b^j)$. 
Hence we have $ab = ba$, a contradiction.  
\end{proof}

Then $a^{-1} c$ is $a^{2i}b^j$ for some integers $i$ and $j$. 
Thus $c = a^{2i+1}b^j$.  
Lemma \ref{power} claims that $(b^m, a^{2i+1}b^j)$ is a generalized torsion pair.

Conversely, assume that $(b^m,c)$ is a generalized torsion pair.  Note $m\ne 0$.
Then $c^{-1}b^mc=b^{-m}$.
Writing $c = a^pb^q$, 
we have $(b^{-q}a^{-p})b^m(a^pb^q)=b^{-m}$, 
which can be reduced to $a^{-p}b^ma^p=b^{-m}$.
Now we show that $p$ must be odd. 
Note that $(b^m)^{a^{p}} = a^{-p} b^m a^{p} = a^{-p+1}(a^{-1} b^m a) a^{p -1} = a^{-p+1} b^{-m} a^{p -1}$. 
Apply this repeatedly to see that if $p$ is odd,
then $(b^m)^{a^{p}} = b^{-m}$, 
and if $p$ is even, 
then $(b^m)^{a^{p}} = b^{m}$. 
Since $b^m \ne b^{-m}$, 
$a^{-p}b^ma^p = (b^m)^{a^{p}} = b^{-m}$ if and only if $p$ is odd. 
\end{proof}

By Lemmas \ref{Kbmap} and \ref{torsion2},
for a given (possibly, non-primitive) generalized torsion pair $(g,c)$ in $\pi_1(X)$,
there exists an immersion $\varphi\colon F\to X$ such that $\varphi(b)=g$ and $\varphi(a)=c$.
It suffices to examine only primitive generalized torsion pairs.

\begin{proof}[Proof of Theorem ~\ref{classification}]
Recall from Claim~\ref{local} that $g$ is conjugate to a generalized torsion element $\varphi(b)$ of $\pi_1(X)$. 
More precisely, up to conjugation we have
\[
\varphi(a)^{-1} \varphi(b) \varphi(a) \varphi(b) = 1\ \mathrm{in}\ \pi_1(X).
\]
Thus $(\varphi(b), \varphi(a))$ is a generalized torsion pair of $\pi_1(X)$. 

As in the proof of Theorem~\ref{g-torsion_order2} we will divide consideration into two cases depending upon $\varphi \colon F \to X$ is vertical or horizontal.

\noindent
\textbf{Case 1}. 
$\varphi \colon F \to X$ is vertical.

We follow the observation in the proof of Lemma~\ref{vertical}. 
Since $\varphi \colon F \to X$ is vertical, 
we have two induced foliations $\mathcal{F}_a$ and $\mathcal{F}_b$ of $F$.

\noindent
\textbf{Subcase 1-1}. 
The induced foliation of $F$ is $\mathcal{F}_a$. 

Then as in the proof of Lemma~\ref{vertical}, $\varphi(a) = t_0$ is an exceptional fiber of even index $2m$, 
$\varphi(a') = t_1$ is also an exceptional fiber of even index $2n$, 
$\varphi(b)$ is a horizontal closed curve satisfying one of the following: 
\begin{enumerate}
\renewcommand{\labelenumi}{(\roman{enumi})}
\item
$p_X(\varphi(b))$ is a non-trivial orientation preserving $x_0$--folded loop based at $x_0 = p_X(t_0)$ if $t_0 = t_1$, or 

\item
$p_X(\varphi(b))$ is an $x_1$--folded loop based at $x_0$ if $t_0 \ne t_1$. 
\end{enumerate}
Note that if $t_0 = t_1$, then $m = n$, 
but if $t_0 \ne t_1$, then $m$ may not be $n$. 
These correspond to generalized torsion pairs described in (i) or (ii) in Theorem~\ref{classification}, respectively.

\noindent
\textbf{Subcase 1-2}. The induced foliation of $F$ is $\mathcal{F}_b$. 

Let $b'$ be a leaf of $\mathcal{F}_b$ close to $b$, which is parallel to $b$. 
Both $\varphi(b)$ and $\varphi(b')$ are Seifert fibers of $X$, 
we see that these are regular fibers. 

On the other hand, $a$ is an orientation reversing curve in $F$, 
and $\varphi(a)$ should be projected to an orientation reversing curve on $B_X$, 
because $X$ is orientable. 
(Thus Subcase 1-2 happens only when $B_X$ is non-orientable. )

Thus $(g, c) = (\varphi(b), \varphi(a))$ is a generalized torsion pair. 

Note that any orientation reversing closed curve on $B_X$ which does not pass through cone points can be $\varphi(a)$. 
In fact, for any orientation reversing closed curve $\gamma$ on $B_X$ which does not pass through cone points, 
$p_X^{-1}(\gamma)$ is a $\pi_1$-injective immersed Klein bottle.
Then for a regular fiber $t$, we have $\gamma^{-1} t \gamma = t^{-1}$, 
equivalently $\gamma^{-1} t \gamma t = 1$, 
and thus we have that 
the regular fiber $t$ in $X$ represents a generalized torsion element of order two.

\noindent
\textbf{Case 2}. 
$\varphi \colon F \to X$ is horizontal.  

By Lemma~\ref{horizontal}, $B_X$ is either $\mathbb{R}P^2$ or the Klein bottle, 
and $(\varphi(b), \varphi(a))$ gives a generalized torsion pair. 

This completes a proof of Theorem~\ref{classification}.
\end{proof}

In Case 2, we do not know if every Seifert fiber space over $\mathbb{R}P^2$ with two exceptional fibers of indices two and 
Seifert fiber space over the Klein bottle with no exceptional fibers admit horizontal immersion from the Klein bottle. 
We give typical examples. 

\begin{example}
We give concrete examples of horizontal Klein bottles.
Let $A=\{(r,\theta) \mid 1\le r\le 2,\ 0\le \theta\le 2\pi\}$ be an annulus, and
let $f\colon A\to A$ be the homeomorphism given by
$f(r,\theta)=(3-r,2\pi-\theta)$.
For $A\times [0,1]$, we identify $A\times \{0\}$ and $A\times \{1\}$ via $f$.
That is, $(r,\theta,1)$ is identified with  $(3-r,2\pi-\theta,0)$.
Then the resulting manifold $M$ is the twisted $I$-bundle over the Klein bottle. 
See Figure~\ref{twisted_I_bundle_1}. 

In $A\times [0,1]$, there are two rectangles
$\{(r,0)\mid 1\le r\le 2\}\times [0,1]$ and $\{(r,\pi)\mid 1\le r\le 2\}\times [0,1]$.
After the identification by $f$, 
these yield two M\"{o}bius bands $B_1$ and $B_2$ in $M$.
Also, the annulus $A\times \{1/2\}\subset A\times [0,1]$ yields a non-separating annulus $C$ in $M$.

It is well known that there are two Seifert fibrations for $M$ \cite{Ja1980}.
One is a Seifert fibration $p_1 \colon M \to B_{X, 1}$, where $B_{X, 1}$ is the M\"obius band with no cone points. 
For this fibration, the annulus $C$ is vertical.
We call this $\mathcal{F}_1$; see Figure~\ref{twisted_I_bundle_1}. 

\begin{figure}[h]
\begin{center}
\includegraphics*[bb=0 0 483 296,width=0.43\textwidth]{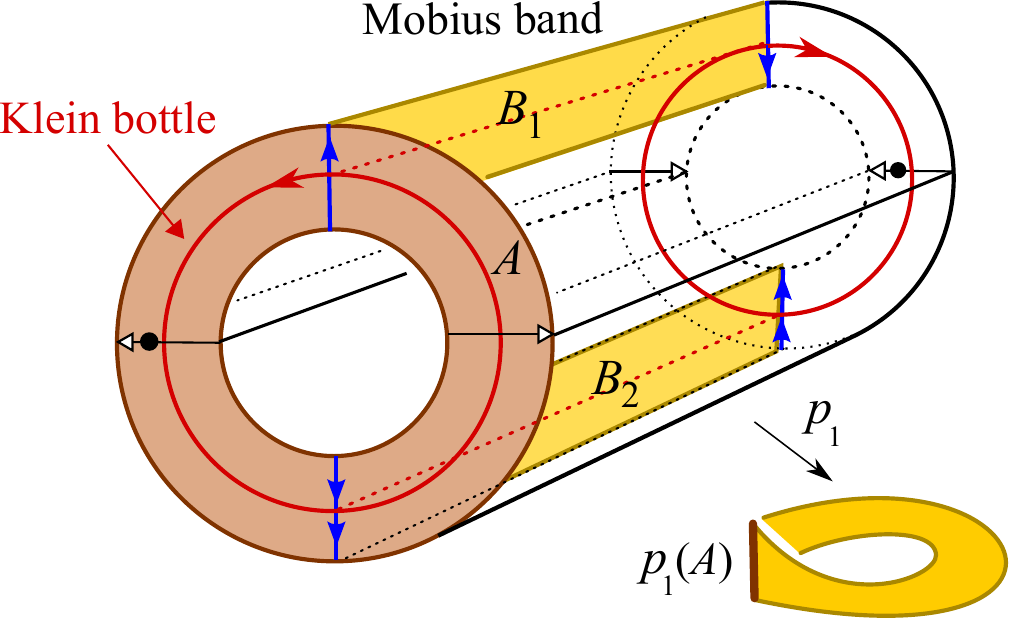}
\end{center}
\caption{Seifert fibration $\mathcal{F}_1$ of the twisted $I$--bundle over the Klein bottle over the M\"obius band.}
\label{twisted_I_bundle_1}
\end{figure}

The other, called $\mathcal{F}_2$, is described as follows.
$A\times [0,1]$ consists of arcs $\{x\}\times [0,1]$.
After the identification by $f$, these arcs yield fibers.
Two arcs $\{(3/2,0)\}\times [0,1]$ and $\{(3/2,\pi)\}\times [0,1]$
yield two exceptional fibers of index two. 
Then we have 
$p_2 \colon M \to B_{X, 2}$, 
where $B_{X, 2}$ is the disk with two cone points of order two; 
see Figure~\ref{twisted_I_bundle_2}.

\begin{figure}[h]
\begin{center}
\includegraphics*[bb=0 0 556 333,width=0.53\textwidth]{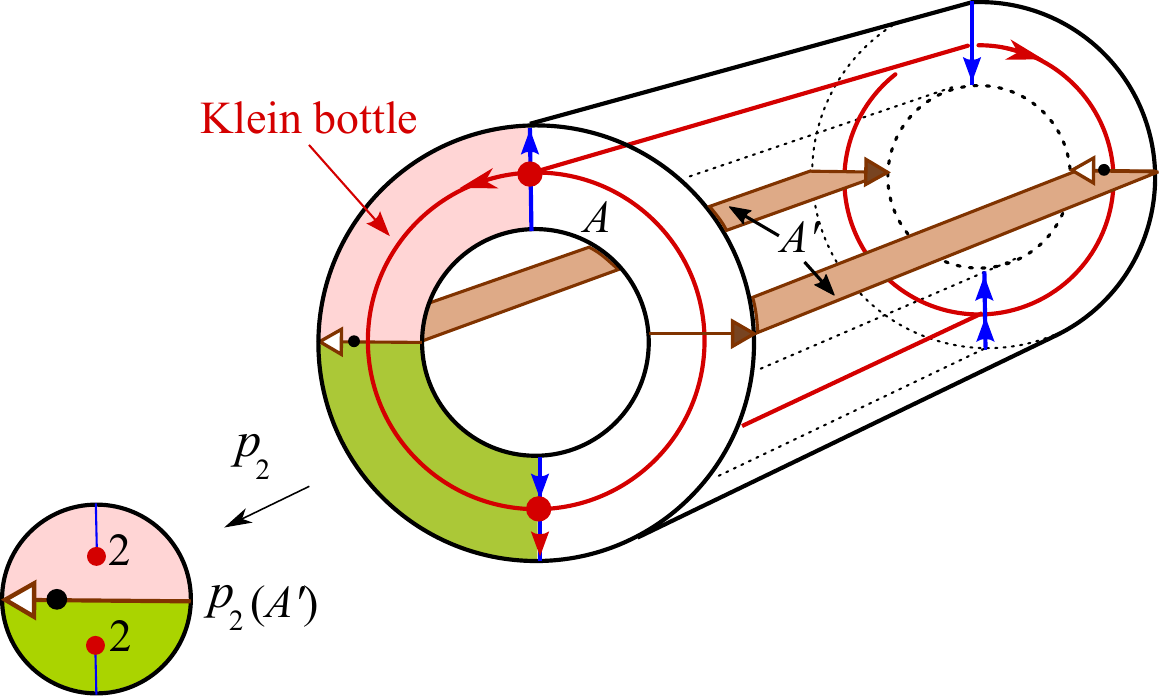}
\end{center}
\caption{Seifert fibration $\mathcal{F}_2$ of the twisted $I$--bundle over the Klein bottle over the the disk with two exceptional fibers of indices two.}
\label{twisted_I_bundle_2}
\end{figure}

First, we assign $\mathcal{F}_1$ to $M$.
From two copies of $M$, we identify those boundary tori via the identity.
The resulting manifold $X$ inherits a Seifert fibration whose base surface is the Klein bottle $B_1\cup B_1$ and
there are no exceptional fibers.
Then $B_1\cup B_1$ and $B_2\cup B_2$ give horizontal embedded Klein bottles in $X$.

Second,
we prepare two copies $M_1$ and $M_2$ of $M$, where $M_i$ is assigned $\mathcal{F}_i$ for $i=1,2$.
Identify $\partial M_1$ and $\partial M_2$ so that
the regular fibers match.
Then $M$ inherits a Seifert fibration whose base surface is the projective plane and
there are two exceptional fibers of index two.

In $M$,  the M\"{o}bius bands $B_1$ and $B_2$ in $M_1$ with the annulus $C\subset M_2$
yield a horizontal embedded Klein bottle.

\end{example}

In the case where $M$ is a hyperbolic $3$--manifold, 
the absence of generalized torsion element of order two may be shown directly using stable commutator length. 
For an independent interest, we give an alternate proof in this direction. 
For the definition of relative stable commutator length, see \cite{CH2024}. 
Recall also that a map $\phi\colon G \rightarrow \mathbb{R}$ is \emph{homogeneous quasimorphism\/} of defect $D(\phi)$ if
\[ D(\phi)=\sup_{g,h \in G}|\phi(gh)-\phi(g)-\phi(h)| < \infty, \quad \phi(g^{k})=k\phi(g) \ (\forall g\in G, k \in \mathbb{Z}).\]
It is known that the value of an element under a homogeneous quasimorphism is determined by its conjugacy class (\cite{Ca2009}).

Bavard's duality is quite useful to estimate the stable commutator length. 

\begin{theorem}[\textbf{Bavard's duality theorem} \cite{Bavard1991}]
\label{Bavard}
For $g \in [G, G]$,
\[ \mathrm{scl}_{G}(g)=\sup_{\phi} \frac{|\phi(g)|}{2D(\phi)} \]
where $\phi\colon G \rightarrow \mathbb{R}$ runs all homogeneous quasimorphisms of $G$ which are not homomorphisms. 
\end{theorem}

Let $M$ be a hyperbolic $3$--manifold.
Suppose that $\pi_1(M)$ contains a generalized torsion element $g$ of order two.
Since $g$ is conjugate to its inverse,
for any homogeneous quasimorphism $\phi\colon \pi_1(M) \to \mathbb{R}$, 
$\phi(g) = \phi(g^{-1}) = - \phi(g)$, i.e. $\phi(g) = 0$. 
Theorem~\ref{Bavard} implies that 
$\mathrm{scl}_{\pi_1(M)}(g)=0$. 
Then \cite[Lemma~2.7 (1)]{CH2024} shows that the relative stable commutator length satisfies
\[
\mathrm{scl}_{(\pi_1(M), \{\pi_1(T_i)\})}(g) \le \mathrm{scl}_{\pi_1(M)}(g) = 0.
\] 
Thus $\mathrm{scl}_{(\pi_1(M), \{\pi_1(T_i)\})}(g) = 0$. 

Lemma 8.13 of \cite{CH2024} claims that
$\pi_1(M)$ has a strong spectral gap relative to the peripheral subgroups, and that
$\mathrm{scl}_{(\pi_1(M), \{ \pi_1(T_i) \} )}(g)=0$ if and only if $g$ is conjugate into some peripheral subgroup $\pi_1(T_i)$. 
Hence if $M$ is closed, then we are done.

Assume that $\partial M \ne \varnothing$.
Then the condition $\mathrm{scl}_{(\pi_1(M), \{ \pi_1(T_i) \} )}(g)=0$ implies that $g$ is conjugate into a peripheral subgroup $\pi_1(T_i)$ for some $i$. 
However, \cite[Corollary 8.17]{CH2024} shows that $\mathrm{scl}_{\pi_1(M)}(g) > 0$, because $g\ne 1$. 
This is a contradiction.

\section{$R$--groups and $\overline{R}$--groups}

There is a handy characterization of an $\overline{R}$--group \cite{FW1999} using a ``Baumslag--Solitar equation'' rather than an isolator. 

\begin{lemma}[\cite{FW1999}]
\label{lem:Rbar}
A group $G$ is an $\overline{R}$--group if and only if
$G$ is an $R$--group and
a Baumslag--Solitar equation $x^{-1}y^mx=y^n$ for $y\ne 1$ implies
$m=n$.
\end{lemma}

A Baumslag--Solitar equation has been extensively studied in  group theory and low-dimensional topology. 
Among them, we recall the following strong result due to Shalen \cite{Sha2001}. 

\begin{lemma}[\cite{Sha2001}]
\label{lem:sha}
Let $M$ be a connected orientable $3$--manifold.
Suppose that $x^{-1}y^mx=y^n$ holds for $x,y\in \pi_1(M)$ and non-zero integers $m$ and $n$.
Then either
\begin{itemize}
\item[(1)]
$y$ has finite order, or
\item[(2)]
$m=\pm n$.
\end{itemize}
\end{lemma}

\begin{lemma}
\label{lem:R-gpair}
An $R$--group does not contain a generalized torsion element of order two.
In particular, it does not have a subgroup isomorphic to the fundamental group
of the Klein bottle.
\end{lemma}

\begin{proof}
This is \cite[Theorem 7]{HMT2023}.
\end{proof}

\begin{proof}[Proof of Theorem \ref{RRbar}]
First, any $\overline{R}$--group is an $R$--group \cite{FW1999}.

Conversely, assume that $G$ is an $R$--group.
Suppose that $G$ is not an $\overline{R}$--group for a contradiction.
By Lemma \ref{lem:Rbar}, 
there exist $x$ and $y\ne 1$ satisfying $x^{-1}y^mx=y^n$ with $m\ne n$.
Since $G$ is torsion-free, Lemma \ref{lem:sha} implies that $m=-n\ (\ne 0)$.
Then $(x^{-1}yx)^m=x^{-1}y^mx=y^{-m}=(y^{-1})^m$.
Since $G$ is an $R$--group, $x^{-1}yx=y^{-1}$.
Thus $y$ is a generalized torsion element of order two, which is impossible by Lemma \ref{lem:R-gpair}.
\end{proof}

Finally, we give a proof of Theorem \ref{R-group_characterization}.
In the proof, we use an easy fact on bi-ordering.
Let us recall that a group $G$ is said to be \textit{bi-orderable\/}
if $G$ admits a strict total ordering which is invariant under left and right multiplications.
We note that a bi-orderable group is torsion-free (see \cite{CR2012}).

\begin{lemma}\label{lem:BO-R}
If $G$ is bi-orderable, then $G$ is an $R$--group.
\end{lemma}

\begin{proof}
Let $x$ and $y$ are distinct elements; without loss of generality we may assume $x < y$. 
We first observe that $x^n \ne y^n$ for any $n > 0$. 
For any $i > 0$, $x^i < y^i$ implies $x^{i+1} = x x^i < y x^i < y y^i = y^{i+1}$. 
Thus $x^n < y^n$ for any $n > 0$. 
Suppose that $x^{-n} = y^{-n}$ for some $n > 0$. 
This shows that $x^n = y^n$, a contradiction. 
\end{proof}

\begin{proof}[Proof of Theorem \ref{R-group_characterization}]

We begin by showing the ``only if'' part. 
Assume that $G$ is an $R$--group. 
Suppose for a contradiction that $M$ has a Seifert fibered decomposing piece $X$ with an exceptional fiber or 
the non-orientable base surface $B_X$. 

Suppose first that $X$ has a non-orientable base surface $B_X$. 
Then $X$ has a vertical Klein bottle, 
which yields a generalized torsion element of order two. 
By Lemma~\ref{lem:R-gpair}, $\pi_1(X)$ is not an $R$--group. 
Hence $G$ is not an $R$--group neither. 

Now assume that $B_X$ is an orientable surface of genus $g$ with $k$ boundary components, where $g \ge 0$ and $k \ge 0$. 
Let $t_1, \dots, t_n$ be exceptional fibers of indices $p_1, \dots, p_n$, respectively. 
By the assumption, we have $n \ge 1$. 
Then $\pi_1(X)$ has a presentation:
\begin{equation}
\begin{split}
\langle a_1, b_1, \dots, a_g, b_g, & c_1, \dots, c_n,  d_1, \dots, d_k, h   \mid   [a_i, h]=[b_i, h] = [c_j, h] = [d_s, h] = 1,\\
& c_i^{p_i} = h^{\alpha_i},\ [a_1, b_1]\cdots[a_g, b_g]c_1\dots c_n d_1\cdots d_k = h^b \rangle,  
\end{split}
\end{equation}
where $(p_i,\alpha_i)$ is the Seifert invariant of the $i$-th exceptional fiber $t_i$, and $b$ is an integer.
If $g=0$, $n \le 2$ and $k = 0$, 
then $X$ is a lens space, and so it is spherical or $S^2 \times S^1$, a contradiction. 
If $g=0$, $n = 1$ and $k  = 1$, 
then it is easy to see that $X$ is a solid torus, contradicting the assumption. 
Thus we have a non-trivial element $y \in \pi_1(X)$ which does not commute with $c_1$. 

It should be noted that a regular fiber represents a central element $h$ of $\pi_1(X)$. 
So the equation $x^{p_1} = c_1^{p_1}$ has solution 
$x = c_1$, $y^{-1} c_1 y$ for the above $y \in \pi_1(X)$. 
In fact, 
\[
(y^{-1} c_1 y)^{p_1} = y^{-1} c_1^{p_1} y = y^{-1} h^{\alpha_1} y = h^{\alpha_1}  = c_1^{p_1}.
\]
Since $[c_1, y] \ne 1$, 
$c_1$ and $y^{-1}c_1y$ give distinct ($p_1$-th) roots of $c_1^{p_1}$,
and hence $\pi_1(X)$ is not an $R$--group, a contradiction again.

Let us prove the ``if'' part. 
Assume that $M$ does not have a Seifert fibered decomposing piece $X$ which has either (i) non-orientable base surface or 
(ii) an exceptional fiber.
We will show that $G$ is an $R$--group. 

Suppose for a contradiction that $G$ is not an $R$--group. 
This means that there exist distinct elements $x$ and $y$ such that $x^n=y^n$ holds for some $n\ne 0$.
Let $z=x^n$.  Then $x$ and $y$ are $n$-th roots of $z$.
We show that $z$ has non-trivial root structure (see \cite{JS1979}).
Suppose that $x$ and $y$ lie in the same cyclic group $\langle t\rangle$ for a contradiction.
Then $x=t^i$ and $y=t^j$ for some integers $i$ and $j$.
Since $x^n=y^n$, we have $t^{in}=t^{jn}$, so $t^{(i-j)n}=1$.
Then $i=j$, because $G$ is torsion-free and $n\ne 0$.
Thus $x=y$, a contradiction.
Hence, $z$ has non-trivial root structure. 

Assume first that $M$ is a Haken $3$--manifold. 
In this case we will apply \cite[Theorem VI.3.1]{JS1979}. 
However, for our purpose we need to confirm that 
$M$ does not contain the embedded Klein bottle; 
see \cite[Addendum to Theorem VI.3.1]{JS1979}.

So before apply \cite[Theorem VI.3.1]{JS1979}, 
we claim that $M$ does not contain the embedded Klein bottle.
Suppose that $M$ contains the Klein bottle $F$.
Then the regular neighborhood $N(F)$ of $F$ gives the twisted $I$--bundle over $F$.
Note that $N(F)$ admits two Seifert fibrations: one over the disk with two exceptional fibers of indices two, or
one over the M\"{o}bius band without exceptional fiber.
Let us write $T=\partial N(F)$.
If $T$ is incompressible in $M$, 
then $M=N(F)$ or $M$ admits a Seifert fibered decomposing piece with an exceptional fiber or
a non-orientable base surface.
This contradicts the assumption on $M$.
Hence $T$ is compressible.
Since $M$ is irreducible, $M$ is the union of $N(F)$ and a solid torus.
Then $M$ is any one of  $S^2\times S^1$,  $\mathbb{R}P^3\#\mathbb{R}P^3$, or a prism manifold.
All of these are impossible from the assumption.
Hence $M$ does not contain the Klein bottle.

Now we may 
apply \cite[Theorem VI.3.1]{JS1979} to see that $M$ has a Seifert fibered decomposing piece $X$ which has an exceptional fiber.  
This contradicts the initial assumption. 

Let us assume that $M$ is a non-Haken $3$--manifold. 
Then $M$ is closed \cite{Hem1976,Ja1980}, 
and since $M$ is irreducible and $G$ is infinite, the Geometrization implies that $M$ is either a Seifert fiber space or a hyperbolic $3$--manifold of finite volume.  
In the case where $M$ is a Seifert fiber space, 
the assumption that its base surface is orientable and it has no exceptional fiber implies
that  $M$ is a circle bundle over an orientable surface $B_M$. 

If $B_M \ne S^2$, then by \cite[Theorem 1.5]{BRW2005}, 
$G$ is bi-orderable. 
Then Lemma~\ref{lem:BO-R} shows that $G$ is an $R$--group, contradicting the assumption. 
Thus we see that $B_M = S^2$.
However, a circle bundle over $S^2$ is either $S^3$, $S^2\times S^1$ or a lens space.
These are all impossible.

Assume that $M$ is a closed hyperbolic $3$--manifold.
It is known that the centralizer of any non-trivial element is cyclic and hence $G$ has the unique root property (see \cite{BB2010}),
and it is an $R$--group, contradicting the assumption. 
\end{proof}

\noindent
\textbf{Acknowledgements.}
We would like to thank Joel Hass for a private communication which enables us to apply his result effectively in this work. 
We also thank the referee for careful reading and useful suggestions.

 The first named auther has been supported by JST, the establishment of university fellowships towards the creation of science technology innovation, Grant Number JPMJFS2129.
The second named author has been partially supported by JSPS KAKENHI Grant Number 19K03502 and Joint Research Grant of Institute of Natural Sciences at Nihon University for 2022.
The third named author has been partially supported by JSPS KAKENHI Grant Number JP20K03587.




\begin{thebibliography}{99}

\bibitem{AFW2015}
M. Aschenbrenner, S. Friedl and H. Wilton,
\textit{$3$--manifold groups},
 EMS Series of Lectures in Mathematics. European Mathematical Society (EMS), Z\"{u}rich, 2015. 


\bibitem{BB2010}
L. Bartholdi and O. Bogopolski,
\textit{On abstract commensurators of groups},
J. Group Theory \textbf{13} (2010), no. 6, 903--922. 

\bibitem{Bavard1991}
C. Bavard; 
Longeur stable des commutateurs,  
Enseign.\ Math.\ \textbf{37} (1991), 109--150. 


\bibitem{BRW2005} 
S. Boyer, D. Rolfsen and B. Wiest,
\textit{Orderable $3$--manifold groups}, 
Ann.\ Inst.\ Fourier \textbf{55} (2005), 243--288. 

\bibitem{Ca2009}
D. Calegari, 
\textit{scl},
MSJ Mem., 20
(Mathematical Society of Japan, Tokyo, 2009).

\bibitem{Ch2018}
L.  Chen,
\textit{Spectral gap of scl in free products},
Proc. Amer. Math. Soc. \textbf{146} (2018), no. 7, 3143--3151. 

\bibitem{CH2024}
L. Chen and N. Heuer,
\textit{Spectral gap of scl in graphs of groups and $3$--manifolds},
preprint. \texttt{arXiv:1910.14146.}


\bibitem{CR2012}
A. Clay and D. Rolfsen, 
\textit{Ordered groups, eigenvalues, knots, surgery and L--spaces},
Math. Proc. Cambridge Philos. Soc. \textbf{152} (2012), no. 1, 115--129. 


\bibitem{DD2024}
A. Das and D. Das,
\textit{Reversibility in the Seifert-fibered spaces},
preprint. \texttt{arXiv:2403.01541.}

\bibitem{FW1999}
T. Fay and G. Walls,
\textit{$\overline{R}$--groups},
J. Algebra \textbf{212} (1999), 375--393. 


\bibitem{Hass1984}
J. Hass, 
\textit{Minimal surfaces in Seifert fiber spaces}, 
Topology Appl.\ \textbf{18} (1984), 145--151. 

\bibitem{HH1985}
J. Hass and J. Hughes, 
\textit{Immersions of surfaces in $3$--manifolds}, 
Topology\ \textbf{24} (1985), 97--112. 

\bibitem{Hem1976}
J. Hempel,
\textit{$3$--manifolds}, 
Ann.\ of Math.\ Studies, No.\ 86 (Princeton University Press, Princeton, NJ, 1976).

\bibitem{HMT2023}
K. Himeno, K. Motegi and M. Teragaito,
\textit{Generalized torsion, unique root property and Baumslag--Solitar relation for knot groups}, 
Hiroshima Math. J. \textbf{53} (2023), 345--358.


\bibitem{IMT_PAMS2019}
T. Ito, K. Motegi and M.Teragaito,
\textit{Generalized torsion and decomposition of $3$--manifolds},
Proc. Amer. Math. Soc. \textbf{147} (2019), no. 11, 4999--5008. 

\bibitem{IMT_BLMS2023}
T. Ito, K. Motegi and M.Teragaito,
\textit{Generalized torsion for hyperbolic $3$--manifold groups with arbitrary large rank},
Bull.\ London Math.\ Soc.,\  \textbf{55} (2023), 1203--1209.


\bibitem{Ja1980}
W. Jaco,
\textit{Lectures on three-manifold topology},
CBMS Regional Conference Series in Mathematics, 43 
(Amer.\ Math.\ Soc., Providence, R.I., 1980).


\bibitem{JS1979}
W. Jaco and P. B. Shalen,
\textit{Seifert fibered spaces in $3$--manifolds},
Mem. Amer. Math. Soc. \textbf{21} (1979), no. 220.

\bibitem{J1979}
K. Johannson,
\textit{Homotopy equivalences of $3$--manifolds with boundaries},
Lecture Notes in Mathematics, 761. Springer, Berlin, 1979.


\bibitem{Kne1929}
H. Kneser,
\textit{Geschlossene Fl\"achen in dreidimensionalen Mannigfaltigkeiten}, 
Jber.\ Deutsch.\ Math.\ Verein.\ \textbf{38} (1929), 248--260. 


\bibitem{MKS2004}
W. Magnus, A. Karrass and D. Solitar; 
\textit{Combinatorial group theory: presentations of groups in terms of generators and relations}, 
Dover books on mathematics, Courier Corporation, 2004. 


\bibitem{Milnor1962}
J. Milnor,
\textit{A unique decomposition theorem for $3$--manifolds},  
Amer.\ J.\ Math.\ \textbf{84} (1962), 1--7.

\bibitem{Price1977}
T. Price, 
\textit{Homeomorphisms of quaternion space and projective plane in four space}, 
J.\ Austral.\ Math.\ Soc.\ \textbf{23} (1977), 112--128. 

\bibitem{Rubin1979}
J. H. Rubinstein, 
\textit{On $3$--manifolds that have finite fundamental group and contain Klein bottles}, 
Trans.\ Amer.\ Math.\ Soc.\ \textbf{251} (1979), 129--137. 

\bibitem{Sc1983}
P. Scott,
\textit{The geometries of $3$--manifolds},
Bull. London Math. Soc. \textbf{15} (1983), no. 5, 401--487.


\bibitem{Sha2001}
P. Shalen, 
\textit{Three-manifolds and Baumslag--Solitar groups},
Topology Appl. \textbf{110} (2001), no. 1, 113--118. 



\end{thebibliography}


\end{document}